\documentclass[11pt,reqno,oneside]{article}

\usepackage{a4wide,amsthm,amsfonts}
\usepackage{latexsym}
\usepackage{amsmath}
\usepackage{amssymb}
\usepackage[numbers]{natbib}
\usepackage{mathrsfs}
\usepackage[utf8]{inputenc}

\theoremstyle{plain}
\newtheorem{Thm}{Theorem}[section]
\newtheorem{Veta}[Thm]{Theorem}
\newtheorem{Lemma}[Thm]{Lemma}

\newtheorem{Dusledek}[Thm]{Corollary}

\theoremstyle{definition}
\newtheorem{Priklad}[Thm]{Example}
\theoremstyle{remark}
\newtheorem{remark}[Thm]{Remark}

\numberwithin{equation}{section}

\newcommand{\E}{\ensuremath{\mathbb{E}\,}}

\newcommand{\R}{\ensuremath{\mathbb{R}}}
\newcommand{\N}{\ensuremath{\mathbb{N}}}

\newcommand{\dif}{\,\mathrm{d}}
\newcommand{\bb}{{V,V^{*}}}

\newcommand\blfootnote[1]{%
  \begingroup
  \renewcommand\thefootnote{}\footnote{#1}%
  \addtocounter{footnote}{-1}%
  \endgroup
}

\begin{document}

\title
{Filtering of Gaussian processes in Hilbert spaces}
\author{V.~Kubelka\thanks{\textit{Faculty of Mathematics and Physics, Charles University,
Sokolovsk\'{a} 83, 186 75 Prague 8, Czech Republic}} \\ \texttt{kubelka@karlin.mff.cuni.cz}
 \and B.~Maslowski\footnotemark[1] \\ \texttt{maslow@karlin.mff.cuni.cz}
 }

\maketitle

\begin{abstract}
Linear filtering problem for infinite-dimensional Gaussian processes is studied, the observation process being finite-dimensional. Integral equations for the filter and for covariance of the error are derived. General results are applied to linear SPDEs driven by Gauss-Volterra process observed at finitely many points of the domain.
\end{abstract}


\section*{Introduction}
The aim of this paper is to study the linear filtering problem for infinite-dimensional Gaussian processes with finite-dimensional observation. Typically, the signal process may be governed by a linear SPDE driven by noise that is not white in time, like a Gaussian Volterra noise or, in particular, fractional Brownian motion (FBM).  
\vspace{1mm}

An analogous problem for finite-dimensional (or scalar) processes have been studied by Kleptsyna and Le Breton \citep{AKK} for the case of general Gaussian process observed through a linear channel driven by standard Brownian motion; this problem has been revisited in \citep{KB2}. A rather general approach to filtering with fractional Brownian motion is presented in \citep{KBR} and specified in more concrete situations, for example, in \citep{KB2}.
\vspace{1mm} 

In infinite dimensions, a pioneering result belongs to Falb \cite{Falb}, where Kalman-Bucy  (KB) type theorem has been established. In this case, both observed process and observation live in a Hilbert space and are governed by linear evolution equation with a $Q$-Wiener process. We are not aware of any analogous result in infinite dimensions for general Gaussian processes that would cover, for example, linear SPDEs driven by fractional noise (however, we would like to point out that a "dual" LQ control problem has been treated, for instance, in \cite{DMP1} and \cite{DMP2}, while related statistical inference problems were addressed in numerous papers, like  \cite{MP1}, \cite{LP} or \cite{KrMa}).

\vspace{1mm}

In the present paper, integral equations for the filter and for covariance of observation error on a rigged Hilbert space are derived and the general results are applied to stochastic parabolic equation perturbed by Gauss-Volterra noise, observed at finitely many points of the domain. In this case, comparing to the classical KB theorem, there are two major obstacles: the fact that the noise does not have independent increments and the need to apply the results to linear SPDEs with space-dependent noise and observation at specific points in the domain of the equation. While the first problem is treated similarly as in the finite-dimensional papers quoted above, the second one is overcome by posing the equation on a rigged Hilbert space. The larger space $H$ (which is usually a Lebesgue space on the domain) is suitable for the definition of the noise term and the stochastic integral, while the smaller space $V$ is contained in the space of continuous functions (for which values at given points are well defined).

\vspace{1mm} 

The paper is divided into four Sections. In section 1 the problem is posed and the main result (Theorem \ref{main_theorem}) is stated and proved. It is shown that the filter and observation error satisfy certain integral equations. In Section 2, uniqueness of solutions to the (nonlinear) integral equation for the error covariance (Theorem \ref{uniqueness_theorem}) is shown. Section 3 is devoted to Examples. At first it is demonstrated that if the signal process is governed by linear evolution equation driven by  a standard cylindrical Wiener process, our result stated in Theorem \ref{main_theorem} reduces to an infinite-dimensional analogue of the classical Kalman-Bucy Theorem (Theorem \ref{theorem_wiener-KB}). Then, following \cite{CoMa}, \cite{CMO} and \cite{CMS}, some basic concepts concerning the infinite-dimensional Gauss-Volterra processes and SPDEs driven by them are recalled. Finally, in Example \ref{final_example} the general results are applied to $2m$-th order stochastic parabolic equation on a bounded domain (and further specified in the case of stochastic heat equation). In Corollary \ref{final_corollary} the main results is specified to the case of pointwise observation of solution to such equation.  

\vspace{1mm}

Bounded linear operator mapping a Banach space $X$ to a a Banach space $Y$ is denoted as $\mathcal{L}(X,Y)$,  $\mathcal{L}(X): = \mathcal{L}(X,X)$. The space of Hilbert-Schmidt operators on a Hilbert space H is denoted by $\mathcal{L}_2 (H)$.

\section{Solution of the filtering problem}
Let $(H,V)$ be a rigged separable Hilbert space, where $H = (H, \langle \cdot {,} \cdot \rangle_H , \| \cdot \|_H )$ and $V = (V, \langle \cdot {,} \cdot \rangle_V , \| \cdot \|_V )$ are separable Hilbert spaces such that $V \subset H$, $V$ is dense in $H$ and identifying $H$ with the dual $H^*$ the embeddings
$$V \hookrightarrow H = H^* \hookrightarrow V^*$$
are continuous and dense. The duality pairing between $V$ and $V^*$ is defined by the inner product on $H$, that is $\langle u,v \rangle_\bb = \langle u,v \rangle_H$ for $u \in V \subset H$ and $v \in H \subset V^*$.

For arbitrary $x, y \in V$ we define tensor product $x \circ y \in \mathcal{L}(V^*,V)$, $(x \circ y)v = x \langle y {,} v \rangle_\bb$, $v \in V^{*}$.

Let us consider stochastic basis $(\Omega , F,P,(F_t))$ and the signal $\theta = \lbrace \theta_t, t \in [0,T] \rbrace$ that is a centered Gaussian mean - square  continuous measurable process in $V$. Let $\xi = \lbrace \xi_t, t \in [0,T] \rbrace$ denote an $\mathbb{R}^{n}$ - valued observation process given as 
\begin{equation}\label{observation_process}
  \xi_t = \int^t_0  A(s) \theta_s \dif s + W_t,
\end{equation} 
where $\left( A(s) \right)_{s \in [0,T]}$ is a family of linear operators $V \rightarrow\mathbb{R}^{n}$ such that mapping $t \rightarrow A(t)$ is strongly measurable and $\| A(t) \|_{\mathcal{L}(V,\mathbb{R}^{n})} \leq K$, $t \in [0,T]$ for some $K < \infty$. Here $W = \lbrace W_t, t \in [0,T]  \rbrace$ is a standard $\mathbb{R}^{n}$ - valued Wiener process independent of the signal $\theta$.

Further, assume that for each $t \in [0,T]$ operator $A(t)$ can be decomposed into functionals $A_1(t),\dots,A_n(t) \in V^*$ such that $A(t)b = (\langle b , A_1(t) \rangle_\bb , \dots, \langle b , A_n(t) \rangle_\bb)^T$ for all $b \in V$. Note that the dual operator $A^{*}(t)$: $\R^n \rightarrow V^{*}$ then satisfies $A^{*}(t)z = \sum_{i=1}^n z_i A_i (t)$ for all $z \in \R^n$.  
   
We are dealing with the optimal filter $\widehat{\theta}_t$, which is defined as
\begin{equation*}
\widehat{\theta}_t = \mathbb{E}[\theta_t| F^{\xi}_{t}],
\end{equation*}
where $( F^{\xi}_{t})_{t \in [0,T]}$ is the filtration generated by the observation process $\xi$.

Set $K^{\theta}(t,s) = \mathbb{E}[ \theta_t \circ \theta_s ]$, $t,s \in [0,T]$. Notice that the mean - square  continuity of the process $\lbrace \theta_t, t \in [0,T] \rbrace$ implies that the mapping $K^{\theta}: [0,T]^2 \rightarrow \mathcal{L}(V^{*},V)$ is strongly continuous and bounded.

\begin{Veta}\label{main_theorem}
Let $\Delta = \lbrace (t,s) \in [0,T]^2 ; 0 \leq s \leq t \leq T \rbrace$.
The filter $\widehat{\theta}$ satisfies the stochastic integral equation
  \begin{equation}\label{main_integral_eq}
  \widehat{\theta}_t = \int^t_0 \Phi(t,s)A^{*}(s) \dif \xi_s -  \int^t_0 \Phi(t,s)A^{*}(s)A(s) \widehat{\theta}_s \dif s, \ \ t \in [0,T],
  \end{equation}
where operator $\Phi$: $\Delta \rightarrow \mathcal{L}(V^*,V)$ defined as $\Phi(t,s) = \mathbb{E}[\theta_t \circ (\theta_s -  \widehat{\theta}_s)]$ for all $(s,t) \in \Delta$ is strongly continuous and satisfies the integral equation
    \begin{equation} \label{main_cov_eq}
    \Phi(t,s) = K^{\theta}(t,s) - \sum_{j=1}^n \int^s_0 \left( \Phi(t,r)A_j(r) \right) \circ \left( \Phi(s,r) A_j(r) \right) \dif r.
    \end{equation}
Moreover, for all $t \in [0,T]$, $\Phi(t,t)$ is the covariance of the estimation error at time  $t \in [0,T]$, that is,
  \begin{equation}\label{main_var_eq}
     \Phi(t,t) = \mathbb{E} \left[(\theta_t -  \widehat{\theta}_t) \circ (\theta_t -  \widehat{\theta}_t) \right]
  \end{equation}
 holds.
\end{Veta}

In the proof of theorem (\ref{main_theorem})the following lemma will be useful.
 \begin{Lemma}\label{rep_lemma}
Let us consider an $\mathbb{R}^n$ - valued process of diffusion type  $\xi = \lbrace \xi_t, t \geq 0 \rbrace$ on probability space $( \Omega , {F}, {P}, ({F}_t) )$ with the differential
\begin{equation}\label{lemma_xi_eq}
 \dif \xi_t = a_t \dif t + \dif W_t,\ \ \ \xi_0 = 0,
\end{equation} 
where  $a = \lbrace a_t, t \geq 0 \rbrace$ is an $(F^{\xi}_t)$ - progresively measurable $\mathbb{R}^n$ - valued process and 
 $W = \lbrace W_t, t \geq 0 \rbrace$ is a standard $(F_t)$ - Wiener process in  $\mathbb{R}^n$.
If  
\begin{equation}\label{lemma_cond}
\int^T_0 \mathbb{E} {\mid a_s \mid}^2 \dif s < \infty
\end{equation} 
 for every $T < \infty$, then
any one - dimensional $(F_t^{\xi})$ - martingale X, forming together with $( \xi, W )$ a Gaussian system can be represented in the form
\begin{equation}\label{M_determ_rep}
X_t = \mathbb{E} X_0 + \sum_{j = 1}^n \int^t_0 f_j (s) \dif W^j_s,
\end{equation}
where $f_1, \dots , f_n$ are deterministic square integrable measurable functions.
\end{Lemma}

\begin{proof}
Note that by (\ref{lemma_xi_eq}) $W$ is also $(F^{\xi}_t)$ - Wiener process.

According to Theorem 3.1 in \citep{FKK} there is a representation
\begin{equation}\label{M_stoch_rep}
X_t = \mathbb{E} X_0 + \sum_{j = 1}^n \int^t_0 f_j (s) \dif W^j_s,
\end{equation}
where $f = (f_1, \dots , f_n)$ is $(F^{\xi}_t)$ - progresively measurable $\R^n$ - valued process and
\begin{equation}\label{f_property}
\int^T_0 \mathbb{E} {\mid f (s) \mid}^2 \dif s  < \infty
\end{equation}
 for every $T < \infty$.
 
Using Gaussianity of the process $(X, \xi, W)$ and Theorem on Normal Correlation (cf. Theorem 13.1 in \cite{LS2}) it can be shown, analogously to the proof of Theorem 5.21 in \cite{LS1}, that for all $j = 1, \dots , n$ and for all $0 \leq s \leq t < \infty$ 
\begin{equation}\label{lem_proof_det_cov}
\mathbb{E} \left[ \left( X_t - X_s \right) \left( W^j_t - W^j_s \right) \mid F^{\xi}_{s} \right] 
= \mathbb{E} \left[ \left( X_t - X_s \right) \left( W^j_t - W^j_s \right) \right]
\end{equation}
holds.

Furter, using It\^o formula and the fact that mixed variation ${\langle  W^j {,} W^k \rangle}_t = 0$ for $j \neq k$ and ${\langle  W^j {,} W^k \rangle}_t = t$ for $j = k$ we have
\begin{align}\label{lem_proof_Ito}
&\left( X_t - X_s \right) \left( W^j_t - W^j_s \right)
=  \int^t_s \left( X_r - X_s \right) \dif  W^j_r +  \sum_{k = 1}^n \int^t_s \left( W^j_r - W^j_s \right) f_k (r) \dif W^k_r + \int^t_s f_j (r) \dif r. 
\end{align}
Now, using (\ref{lem_proof_det_cov}), (\ref{lem_proof_Ito}) and martingale property of It\^o integral for all $j = 1, \dots , n$ we get 
\begin{equation*}
\mathbb{E} \left[ \int^t_s f_j (r) \dif r \right]
= \mathbb{E} \left[ \int^t_s f_j (r) \dif r  \mid F^{\xi}_{s} \right].
\end{equation*}
Taking into account (\ref{f_property}) we can use Fubini's theorem to obtain
\begin{equation}\label{lem_proof_condexp=exp}
 \int^t_s \mathbb{E} \left[f_j (r) \right] \dif r
= \int^t_s \mathbb{E} \left[ f_j (r) \mid F^{\xi}_{s} \right] \dif r 
\end{equation}
for all $j = 1, \dots , n$.

We can use (\ref{lem_proof_condexp=exp}) piecewise on a sequence of decompositions
 $\lbrace 0 = t_0^{(n)} < \cdots < t_n^{(n)} = t, n \in \N \rbrace$ of interval $[0,t]$ such that 
$\max_{i= 1, \ldots n} \mid t_{i}^{(n)} - t_{i-1}^{(n)} \mid \xrightarrow{n \rightarrow \infty} 0$
to obtain
\begin{equation*}
\int_0^t \mathbb{E} \left[f_j (r) \right] \dif r = \int_0^t f_{j,n} (r) \dif r, \quad j = 1, \dots , n,
\end{equation*}
where $f_{j,n} (r) = \mathbb{E} \left[ f_j (r) \mid F^{\xi}_{t_i^{(n)}} \right] $, $t_i^{(n)} \leq r < t_{i+1}^{(n)}$.

Analogously to the proof of Theorem 5.21 in \cite{LS1}, using continuity of the filtration $ (F^{\xi}_{s},\ 0 < s < t)$ which follows from Theorem 5.19 in \cite{LS1} and the uniform integrability of $\lbrace f_{j,n}, n \in \N \rbrace$  we get
\begin{equation*}
 \E \mid \int^t_s \lbrace \E \left[f_j (r) \right] - f_j (r) \rbrace \dif r \mid \xrightarrow{n \rightarrow \infty} 0, \quad j = 1, \dots , n.
\end{equation*}
From this, for all $j = 1, \dots , n$ and each $t,\ 0 \leq t \leq T$, we have
\begin{equation*}
\int^t_s \mathbb{E} \left[f_j (r) \right] \dif r
= \int^t_s  f_j (r) \dif r
\end{equation*}
and, therefore, for all $j = 1, \dots , n$ and for almost all  $t,\ 0 \leq t \leq T$
\begin{equation}\label{lem_proof_exp=determ}
 f_j (t) = \mathbb{E} \left[f_j (t) \right].
\end{equation}
Equality (\ref{lem_proof_exp=determ}) together with (\ref{M_stoch_rep}) proves the representation (\ref{M_determ_rep}).
\end{proof}
\vspace{1cm}

Now, let us prove of theorem (\ref{main_theorem}).

\begin{proof}
According to Lemma 2.2 in \citep{FKK} process $\lbrace  \tilde{W}_t, t \in [0,T] \rbrace$ defined as 
\begin{equation}\label{inov_process}
\tilde{W}_t = {\xi}_t -  \int^t_0 \mathbb{E} [ A(r) \theta_r \mid F_t^{\xi} ] \dif r
= {\xi}_t -  \int^t_0 A(r) \widehat{\theta}_r \dif r.
\end{equation}
is $\mathbb{R}^{n}$ - valued $( F^{\xi}_{t} )$ - standard Wiener process called innovation process. The formula (\ref{inov_process}) reads
\begin{equation}\label{observ_inov_eq}
\dif \xi_t = A(t) \widehat{\theta}_t \dif t + \dif \tilde{W}_t,\ \ \ \xi_0 = 0
\end{equation}
and the process $\lbrace A(t) \widehat{\theta}_t,\ t \geq 0 \rbrace$ satisfies the condition (\ref{lemma_cond}) of Lemma \ref{rep_lemma} hence the observation process takes the form (\ref{lemma_xi_eq}).

Further, define the square integrable $V$ - valued process $M^s =  \lbrace M^s_t, t \in [0,T] \rbrace$ as 
\begin{equation}\label{M_def}
M^s_t = \mathbb{E}[\theta_s| F^{\xi}_{t}]
\end{equation}
for all $s \in [0,T]$.
Note that $\widehat{\theta}_t = M^t_t$ and the proces $M^s$ is $(F^{\xi}_t)$ - martingale.

Let $\lbrace e_i, i \in I \rbrace$ be an orthonormal basis on $V$. Then we have
\begin{equation}\label{Ms_rep}
M^s_t = \sum_{i \in I} e_i M^s_{i, t} 
\end{equation}
where $M^s_{i, t} = \mathbb{E}[\langle \theta_s {,} e_i \rangle_H | F^{\xi}_{t}]$. Process $M^s_i$ is one - dimensional $( F^{\xi}_t )$ - martingale for all $s \in [0,T]$ and all $i \in I$. The triple $(M^s_i, \xi, \tilde{W})$ forms a Gaussian system, therefore, in virtue of Lemma \ref{rep_lemma}
\begin{equation}\label{Msi_rep}
M^s_{i, t} = \sum_{j = 1}^n \int^t_0 F^s_{i,j} (r) \dif \tilde{W}^j_r,
\end{equation}
where $F^s_{i,1}, \dots , F^s_{i,n}$ are deterministic square integrable measurable functions, holds for every $s \in [0,T]$ and every $i \in I$.
From equations (\ref{Ms_rep}) and (\ref{Msi_rep}) we obtain
\begin{equation}\label{Ms_summ}
M^s_t = \sum_{i \in I} e_i \sum_{j = 1}^n \int^t_0 F^s_{i,j} (r) \dif \tilde{W}^j_r. 
\end{equation}
Using It\^o isometry
\begin{align*}
\infty &> \mathbb{E} \| M^s_t \|^2_{V}  = \sum_{i \in I} \mathbb{E} ( M^s_{i,t} )^2 = \sum_{i \in I} \mathbb{E} \left[    
\sum_{j = 1}^n \int^t_0 F^s_{i,j} (r) \dif \tilde{W}^j_r \right]^2
 = \sum_{i \in I} \sum_{j = 1}^n \int^t_0  \left( F^s_{i,j} (r) \right)^2 \dif r \\
 &= \sum_{j = 1}^n \int^t_0 \sum_{i \in I} \left( F^s_{i,j} (r) \right)^2 \dif r
 =  \sum_{j = 1}^n \| F^s_j \|_{L^2(V)},
\end{align*}
where $F^s(r) = (F^s_1(r), \dots , F^s_n(r))$ is square integrable $V^n$ - valued deterministic function such that $F^s_j (r) = \sum_{i \in I} e_i  F^s_{i,j} (r)$ for all $j = 1, \dots , n$.

Therefore, swaping the sums and the integral in (\ref{Ms_summ}) we finally obtain
\begin{equation}\label{Ms_characterization}
M^s_t = \int^t_0 F^s (r) \dif \tilde{W}_r = \sum_{j = 1}^n \int^t_0 F^s_j (r) \dif \tilde{W}^j_r.
\end{equation}

Now, for every $s \in [0,T]$, we can consider an arbitrary square integrable measurable $V^n$ - valued deterministic function  $f^s(r) = (f^s_1 (r), \dots , f^s_n (r))$ to show
\begin{equation}\label{zero_op_eq}
\mathbb{E}\left[ \left(\theta_s - M^s_t \right) \circ {\left(\int^t_0 f^s (r) \dif \tilde{W}_r \right)} \right] = 0.
\end{equation}

Indeed, using (\ref{M_def}) for arbitrary $v \in V^{*}$ we have
\begin{align*}
& \mathbb{E}\left[ \left(\theta_s - M^s_t \right) \circ {\left(\int^t_0 f^s (r) \dif \tilde{W}_r \right)} \right] v 
= \mathbb{E}\left[ \left(\theta_s - M^s_t \right) \langle \int^t_0 f^s (r) \dif \tilde{W}_r {,} v \rangle_\bb \right] \\
&=\mathbb{E}\left[ \mathbb{E}\left[ \left(\theta_s - M^s_t \right)  \langle \int^t_0 f^s (r) \dif \tilde{W}_r {,} v \rangle_\bb \mid  F^{\xi}_{t} \right] \right] \\
&= \mathbb{E}\left[ \left( \mathbb{E}\left[ \theta_s \mid  F^{\xi}_{t} \right] - M^s_t \right)  \langle \int^t_0 f^s (r) \dif \tilde{W}_r {,} v \rangle_\bb \right] = 0.
\end{align*}

Further, we show that
\begin{equation}\label{It_izometrie}
\mathbb{E}\left[ M^s_t \circ {\left(\int^t_0 f^s (r) \dif \tilde{W}_r \right)} \right] =
\sum_{j = 1}^n \int^t_0 F^s_j (r) \circ  {f^s_j (r)} \dif r.
\end{equation}
Note that to prove equality of two arbitrary operators $P,Q \in \mathcal{L}(V^*,V)$ it is sufficient to show $ \langle P v {,} b \rangle_V = \langle Q v {,} b \rangle_V$ for all elements $v \in V^{*}$, $b \in V$. We have that
\begin{align*}
& \langle \mathbb{E}\left[ M^s_t \circ {\left(\int^t_0 f^s (r) \dif \tilde{W}_r \right)} \right] v {,} b \rangle_V
= \langle \mathbb{E}\left[ M^s_t  \langle \int^t_0 f^s (r) \dif \tilde{W}_r {,} v \rangle_\bb \right] {,} b \rangle_V \\
&= \mathbb{E}\left[ \langle  M^s_t  \langle  \int^t_0 f^s (r) \dif \tilde{W}_r {,} v \rangle_\bb  {,} b \rangle_V \right]
= \mathbb{E}\left[  \langle  \int^t_0 f^s (r) \dif \tilde{W}_r {,} v \rangle_\bb   \langle  M^s_t  {,} b \rangle_V \right] \\
&= \mathbb{E}\left[  \langle \int^t_0 f^s (r) \dif \tilde{W}_r {,} v \rangle_\bb   \langle \int^t_0 F^s (r) \dif \tilde{W}_r {,}  b \rangle_V \right] \\
&= \mathbb{E}\left[  \langle \sum_{j = 1}^n \int^t_0 f^s_j (r) \dif \tilde{W}^j_r {,} v \rangle_\bb   \langle \sum_{j = 1}^n \int^t_0 F^s_j (r) \dif \tilde{W}^j_r  {,} b \rangle_V \right] \\
&= \mathbb{E}\left[  \sum_{j = 1}^n \int^t_0 \langle  f^s_j (r) {,} v  \rangle_\bb  \dif \tilde{W}^j_r \sum_{j = 1}^n \int^t_0 \langle F^s_j (r) {,} b \rangle_V \dif \tilde{W}^j_r \right].
\end{align*}
Using  It\^o isometry and the fact that mixed variation $ \langle \tilde{W}^i {,} \tilde{W}^j \rangle = 0,\ i \neq j$ we obtain
\begin{align*}
&\mathbb{E}\left[  \sum_{j = 1}^n \int^t_0 \langle f^s_j (r) {,} v \rangle_\bb  \dif \tilde{W}^j_r \sum_{j = 1}^n \int^t_0 \langle  F^s_j (r) {,} b \rangle_V \dif \tilde{W}^j_r \right] \\
&= \sum_{j = 1}^n \int^t_0 \langle  f^s_j (r) {,} v \rangle_\bb  \langle  F^s_j (r) {,} b \rangle_V \dif r 
= \sum_{j = 1}^n \langle \left( \int^t_0 F^s_j (r) \circ {f^s_j (r)} \dif r \right) v {,}  b \rangle_V \\
&= \langle \left( \sum_{j = 1}^n \int^t_0 F^s_j (r) \circ {f^s_j (r)} \dif r \right) v {,}  b \rangle_V 
\end{align*}
which concludes the proof of equality (\ref{It_izometrie}).

Next, we show equality
\begin{equation}\label{eq_with_Riesz}
\mathbb{E}\left[ \theta_s \circ {\left(\int^t_0 f^s (r) \dif \tilde{W}_r \right)} \right] 
=  \sum_{j = 1}^n \int^t_0 \left( \mathbb{E}\left[ \theta_s \circ {( \theta_r -  \widehat{\theta}_r )} \right] A_j (r) \right) \circ {f^s_j (r)} \dif r.
\end{equation}
Using (\ref{observation_process}), (\ref{inov_process}) and independence of $\theta$ and $W$ we have 
\begin{align*}
&\mathbb{E}\left[ \theta_s \circ {\left(\int^t_0 f^s (r) \dif \tilde{W}_r \right)} \right] 
= \mathbb{E}\left[ \theta_s \circ {\left(\int^t_0 f^s (r) \dif \xi_r - \int^t_0 f^s (r) A(r) \widehat{\theta}_r \dif r \right)} \right] \\
&= \mathbb{E}\left[ \theta_s \circ {\left(\int^t_0 \sum_{j=1}^n f^s_j (r) \langle \theta_r ,  A_j(r) \rangle_\bb \dif r + \int^t_0 f^s (r) \dif W_r - \int^t_0 \sum_{j=1}^n f^s_j (r) \langle \widehat{\theta}_r ,  A_j(r) \rangle_\bb \dif r \right)} \right] \\
&= \mathbb{E}\left[ \theta_s \circ {\left(\int^t_0 \sum_{j = 1}^n f^s_j (r) \langle  \theta_r - \widehat{\theta}_r {,} A_j (r) \rangle_\bb \dif r \right)} \right]\\
&= \sum_{j = 1}^n \int^t_0 \mathbb{E}\left[ \theta_s \circ {\left(  f^s_j (r) \langle \theta_r - \widehat{\theta}_r {,} A_j (r) \rangle_\bb \right)} \right]\dif r.
\end{align*}
To complete the proof of equality (\ref{eq_with_Riesz}) it is sufficient to show
\begin{equation}\label{helper_equality_for_riesz_eq}
\theta_s \circ {\left(  f^s_j \langle(r) \theta_r - \widehat{\theta}_r {,}  A_j (r) \rangle_\bb \right)}
= \left[ \left( \theta_s \circ ( \theta_r - \widehat{\theta}_r ) \right) A_j (r) \right] \circ {f^s_j (r) }, \ \ j = 1, \dots , n. 
\end{equation}
By the definition of tensor product we obtain
\begin{align*}
&\langle \theta_s \circ {\left(  f^s_j (r) \langle \theta_r - \widehat{\theta}_r {,} A_j(r) \rangle_\bb \right)} v {,} b \rangle_V
= \langle \theta_s \langle  f^s_j (r) \langle \theta_r - \widehat{\theta}_r {,} A_j (r) \rangle_\bb {,} v \rangle_\bb  {,} b \rangle_V \\
&= \langle \theta_r - \widehat{\theta}_r {,} A_j(r) \rangle_\bb  \langle  f^s_j (r){,} v \rangle_\bb \langle \theta_s   {,} b \rangle_V
\end{align*}
for all $v \in V^{*}$, $b \in V$.
Similarly we get
\begin{align*}
&\langle \left( \left[ \left( \theta_s \circ ( \theta_r - \widehat{\theta}_r ) \right) A_j (r) \right] \circ {f^s_j (r) } \right) v {,} b \rangle_V
= \langle  \left( \theta_s \circ ( \theta_r - \widehat{\theta}_r ) \right) A_j (r) \langle f^s_j (r) {,}  v \rangle_\bb {,} b \rangle_V \\
&= \langle f^s_j (r) {,} v \rangle_\bb \langle \theta_s \langle \theta_r - \widehat{\theta}_r {,} A_j (r) \rangle_\bb  {,} b \rangle_V  
= \langle \theta_r - \widehat{\theta}_r {,} A_j (r) \rangle_\bb  \langle  f^s_j (r){,} v \rangle_\bb \langle \theta_s   {,} b \rangle_V.
\end{align*}
Which proves equality (\ref{helper_equality_for_riesz_eq}) and, therefore, completes the proof of (\ref{eq_with_Riesz}).

Combining equalities (\ref{zero_op_eq}), (\ref{It_izometrie}) and (\ref{eq_with_Riesz}) we obtain
\begin{equation}\label{main_eq_proof1}
\sum_{j = 1}^n \int^t_0 F^s_j (r) \circ f^s_j (r) \dif r
= \sum_{j = 1}^n \int^t_0 \left( \mathbb{E}\left[ \theta_s \circ ( \theta_r -  \widehat{\theta}_r ) \right] A_j (r) \right) \circ {f^s_j (r)} \dif r.
\end{equation}
The formula (\ref{main_eq_proof1}) holds for any arbitrary square integrable $V^n$ - valued deterministic function  $f^s(r) = (f^s_1 (r), \dots , f^s_n (r))$ hence
\begin{equation}\label{Fj_characterization}
F^s_j (r) =  \mathbb{E}\left[ \theta_s \circ {( \theta_r -  \widehat{\theta}_r )} \right] A_j (r) = \Phi(s,r) A_j (r)
\end{equation}
for all $s \in [0,T]$ and for almost all $r \in [0,T]$.

From (\ref{Ms_characterization}) and (\ref{Fj_characterization}), by the choice $s=t$ and taking into account (\ref{inov_process}) we have 
\begin{align*}
&\widehat{\theta}_t = M^t_t =  \sum_{j = 1}^n \int^t_0 F^t_j (r) \dif \tilde{W}^j_r
= \sum_{j = 1}^n \int^t_0 \Phi(t,r) A_j (r) \dif \tilde{W}^j_r \\
&= \sum_{j = 1}^n \int^t_0 \Phi(t,r)A_j(r) \dif \xi^j_r -  \sum_{j = 1}^n \int^t_0 \Phi(t,r)A_j(r) \langle \widehat{\theta}_r , A_j(r) \rangle_\bb \dif r \\
&= \int^t_0 \Phi(t,s)A^{*}(s) \dif \xi_s -  \int^t_0 \Phi(t,s)A^{*}(s)A(s) \widehat{\theta}_s \dif s
\end{align*}
which concludes the proof of stochastic integral equation (\ref{main_integral_eq}).

Let us verify the formula (\ref{main_cov_eq}). Notice that
\begin{equation}\label{cov_eq_proof1}
\Phi(t,s) = \mathbb{E} \left[\theta_t \circ (\theta_s -  \widehat{\theta}_s) \right] 
= \mathbb{E} \left[ \theta_t \circ \theta_s \right] -  \mathbb{E} \left[ \theta_t  \circ \widehat{\theta}_s \right].
\end{equation}
Using the above proved representation of $\widehat{\theta}_s$, (\ref{observation_process}), (\ref{inov_process}) and independence of $\theta$ and $W$, we have
\begin{align}
&\mathbb{E} \left[ \theta_t \circ \widehat{\theta}_s \right]
= \mathbb{E} \left[ \theta_t \circ {\left( \int^s_0 \Phi(s,r)A^{*}(r) \dif \xi_r -  \int^s_0 \Phi(s,r)A^{*}(r)A(r) \widehat{\theta}_r \dif r \right) } \right]  \nonumber \\
&= \mathbb{E} \left[ \theta_t \circ \left( \int^s_0 \Phi(s,r) A^* (r) A(r) \theta_r \dif r + \int^s_0 \Phi(s,r) A^* (r) \dif W_r - \int^s_0 \Phi(s,r) A^* (r) A(r) \widehat{\theta}_r \dif r \right) \right]  \nonumber \\
&= \mathbb{E} \left[ \theta_t \circ {\left( \int^s_0 \Phi(s,r) A^* (r) A(r) ( \theta_r - \widehat{\theta}_r ) \dif r \right) } \right] 
= \int^s_0  \mathbb{E} \left[ \theta_t \circ {\left( \Phi(s,r) A^* (r) A(r) ( \theta_r - \widehat{\theta}_r ) \right) } \right] \dif r \nonumber \\
&= \int^s_0  \mathbb{E} \left[ \sum_{j = 1}^n  \theta_t  \circ {\left[ \Phi(s,r) A_j(r) \langle  \theta_r - \widehat{\theta}_r {,} A_j(r) \rangle_\bb \right] } \right] \dif r. \label{cov_proof_helper_eq}
\end{align}
For arbitrary $v \in V^{*}$, $b \in V$ we have that
\begin{align*}
&\langle\mathbb{E} \left[ \sum_{j = 1}^n  \theta_t  \circ {\left[ \Phi(s,r) A_j(r) \langle  \theta_r - \widehat{\theta}_r {,} A_j(r) \rangle_\bb \right] } \right]   v {,} b \rangle_V \\
&=  \mathbb{E} \left[ \sum_{j = 1}^n \langle \theta_r - \widehat{\theta}_r , A_j(r) \rangle_\bb \langle  \Phi(s,r) A_j(r) , v \rangle_\bb \langle \theta_t , b \rangle_V \right] \\ 
&=   \mathbb{E} \left[ \sum_{j = 1}^n  \langle \theta_t \langle \theta_r - \widehat{\theta}_r , A_j(r) \rangle_\bb  , b \rangle_V \langle  \Phi(s,r) A_j(r) , v \rangle_\bb \right] \\
&= \langle  \sum_{j = 1}^n   \mathbb{E} \left[ \theta_t \circ \theta_r - \widehat{\theta}_r \right] A_j(r)  \langle  \Phi(s,r) A_j(r) , v \rangle_\bb  , b \rangle_V \\
&= \langle  \sum_{j = 1}^n  \left[ \Phi(t,r)  A_j(r) \circ \Phi(s,r) A_j(r) \right] v  , b \rangle_V.
\end{align*}
By (\ref{cov_proof_helper_eq}) it follows that
\begin{align}\label{strong_continuity_helper_label}
&\mathbb{E} \left[ \theta_t \circ \widehat{\theta}_s \right]
= \sum_{j = 1}^n \int^s_0 \Phi(t,r)  A_j(r) \circ \Phi(s,r) A_j(r) \dif r.
\end{align}
This, together with (\ref{cov_eq_proof1}), proves the formula (\ref{main_cov_eq}).

By (\ref{cov_eq_proof1}) the mapping $t \mapsto \Phi(t,s)$ is strongly continuous on $[0,T]$ uniformly with respect to $s \in [0,t]$ hence the representation (\ref{strong_continuity_helper_label}) implies the mean - square continuity of the proces $\hat{\theta}$ in $V$. By (\ref{strong_continuity_helper_label}) and (\ref{cov_eq_proof1}) the mapping $\Phi$: $\Delta \rightarrow \mathcal{L}(V^*,V)$ is strongly continuous.

It remains to prove equality (\ref{main_var_eq}). For every $v \in V^{*}$ it holds
\begin{align*}
&\mathbb{E} \left[ \widehat{\theta}_t \circ (\theta_t -  \widehat{\theta}_t) \right]v
= \mathbb{E} \left[ \mathbb{E} \left[ \widehat{\theta}_t \langle \theta_t -  \widehat{\theta}_t {,} v \rangle_\bb \mid F^{\xi}_{t} \right] \right] \\
&= \mathbb{E} \left[ \widehat{\theta}_t \left( \langle \mathbb{E} \left[ \theta_t \mid F^{\xi}_{t} \right] {,} v  \rangle_\bb  - \langle  \widehat{\theta}_t {,} v \rangle_\bb \right) \right] = 0,
\end{align*}
 therefore, we obtain
\begin{align*}
\mathbb{E} \left[(\theta_t -  \widehat{\theta}_t) \circ (\theta_t -  \widehat{\theta}_t) \right] 
= \Phi(t,t) + \mathbb{E} \left[ \widehat{\theta}_t \circ (\theta_t -  \widehat{\theta}_t) \right] = \Phi(t,t)
\end{align*}
which completes the proof of Theorem \ref{main_theorem}.
\end{proof}

\section{Uniqueness of the covariance equation}
In the present section we prove uniqueness of solutions to the integral equation (\ref{main_cov_eq}) in the class of operators $$\mathbb{M} = \lbrace \Psi: \Delta \rightarrow \mathcal{L}(V^*,V);\ \Psi \mbox{ strongly continuous};\ \Psi (t,t) \geq 0\ \forall t \in [0,T] \rbrace.$$

More generally, we prove the following.

\begin{Veta}\label{uniqueness_theorem}
Let $K: \Delta \rightarrow \mathcal{L}(V^*,V)$ be a strongly continuous mapping such that $K(t,t) \geq 0$ for each $t \in [0,T]$.

Then the integral equation
\begin{equation} \label{uniqueness_cov_eq}
\Psi(t,s) = K(t,s) - \sum_{j=1}^n \int^s_0 \left( \Psi(t,r)A_j(r) \right) \circ \left( \Psi(s,r) A_j(r) \right) \dif r,\ \ (t,s) \in \Delta
\end{equation}
has at most one solution in the class $\mathbb{M}$.
\end{Veta}

\begin{proof}
First, we show that every $\Psi \in \mathbb{M}$ satisfying the equation (\ref{uniqueness_cov_eq}) is bounded on $\Delta$.

Using Cauchy - Schwarz inequality we obtain for arbitrary $x,y \in V^*$
\begin{align}
& | \langle \left(\sum_{j=1}^n \int^s_0 \left( \Psi(t,r)A_j(r) \right) \circ \left( \Psi(s,r) A_j(r) \right) \dif r \right) x {,} y \rangle_\bb| \nonumber\\
& \leq \sum_{j=1}^n \int^s_0 | \langle \left[ \left( \Psi(t,r)A_j(r) \right) \circ \left( \Psi(s,r) A_j(r) \right) \right] x {,} y \rangle_\bb| \dif r \nonumber\\
&  = \sum_{j=1}^n \int^s_0 | \langle \Psi(t,r)A_j(r) \langle \Psi(s,r) A_j(r) , x \rangle_\bb {,} y \rangle_\bb| \dif r  \nonumber\\
&  = \sum_{j=1}^n \int^s_0 | \langle \Psi(t,r)A_j(r) {,} y \rangle_\bb| \ | \langle \Psi(s,r) A_j(r) , x \rangle_\bb | \dif r \nonumber\\
& \leq \sum_{j=1}^n  \left( \int^s_0 \left( \langle \Psi(t,r)A_j(r) {,} y \rangle_\bb \right)^2  \dif r \right)^{\frac{1}{2}} \left( \int_0^s \left( \langle \Psi(s,r) A_j(r) , x \rangle_\bb \right)^2 \right)^{\frac{1}{2}} \nonumber\\
& = \sum_{j=1}^n  \left( \int^s_0 \langle \left[ \left( \Psi(t,r)A_j(r) \right) \circ \left( \Psi(t,r) A_j(r) \right) \right] y {,}y \rangle_\bb \dif r  \right)^{\frac{1}{2}} \nonumber\\
& \ \ \ \ \ \ \ \ \ \ \ \ \ \ \ \ \ \ \ \left( \int^s_0 \langle \left[ \left( \Psi(s,r)A_j(r) \right) \circ \left( \Psi(s,r) A_j(r) \right) \right] x {,} x \rangle_\bb \dif r \right)^{\frac{1}{2}}. \label{Pr2_second_term_podruhe}
\end{align}
In the last inequality we can increase the upper bound of the first integral from $s$ to $t$ because the integrand is nonnegative. Thus all we need is to estimate the term
 $$\int^t_0 \langle \left[ \left( \Psi(t,r)A_j(r) \right) \circ \left( \Psi(t,r) A_j(r) \right) \right] x {,} x \rangle_\bb \dif r $$
 for all $j = 1, \ldots , n$ and $x \in V^*$. Using (\ref{uniqueness_cov_eq}) and the boundedness of the family $\left( K(t,s),\ t,s \in [0,T] \right)$ in $\mathcal{L}(V^*,V)$ (which follows by the Resonance Theorem) we obtain
\begin{align}
&0 \leq \int^t_0 \langle \left[ \left( \Psi(t,r)A_j(r) \right) \circ \left( \Psi(t,r) A_j(r) \right) \right] x {,} x \rangle_\bb \dif r \nonumber\\
&\leq \sum_{j=1}^n \int^t_0 \langle \left[ \left( \Psi(t,r)A_j(r) \right) \circ \left( \Psi(t,r) A_j(r) \right) \right] x {,} x \rangle_\bb \dif r
= \langle K(t,t) x {,} x \rangle_\bb - \langle \Psi(t,t) x {,} x \rangle_\bb \nonumber\\
&\leq \langle K(t,t) x {,} x \rangle_\bb \leq C_1(T) {\| x \|}^2_{V^*}, \ \ t \in [0,T]. \label{Pr2_second_term_potreti}
\end{align}
Now, by (\ref{uniqueness_cov_eq}), (\ref{Pr2_second_term_podruhe}), (\ref{Pr2_second_term_potreti}) and again by the boundedness of the family $\left( K(t,s),\ t,s \in [0,T] \right)$ we have that
\begin{align*}
&|\langle \Psi(t,s) x {,} y \rangle_\bb|  \\
& \leq \| K(t,s)\|_{\mathcal{L}(V^*,V)} \|x\|_{V^*} \|y\|_{V^*}  + | \langle \left(\sum_{j=1}^n \int^s_0 \left( \Psi(t,r)A_j(r) \right) \circ \left( \Psi(s,r) A_j(r) \right) \dif r \right) x {,} y \rangle_\bb| \\
& \leq C_2 (T) {\| x \|}_{V^*} {\| y \|}_{V^*}
\end{align*}
for all $x,y \in V^*$ and $(t,s) \in \Delta$, which proves that the family of operators $\left( \Psi (t,s),\ (t,s) \in \Delta \right)$ is bounded in $\mathcal{L}(V^*,V)$.

Assume that $\Psi_1, \Psi_2 \in \mathbb{M}$ are solutions to the equation (\ref{uniqueness_cov_eq}) and using essential supremum set $\varphi: \ [0,T] \rightarrow [0,\infty)$, $\varphi (s) = \sup_{t \in [s,T]} {\| \Psi_1(t,s) - \Psi_2(t,s) \|}_{\mathcal{L}(V^*,V)}$. In virtue of (\ref{uniqueness_cov_eq}) we have

\begin{align*}
&\varphi (s) = \sup_{t \in [s,T]} {\| \Psi_1(t,s) - \Psi_2(t,s) \|}_{\mathcal{L}(V^*,V)} \\
&= \sup_{t \in [s,T]} {\|  \sum_{j=1}^n \int^s_0 \left( \Psi_1(t,r)A_j(r) \right) \circ \left( \Psi_1(s,r) A_j(r) \right)  - \left( \Psi_2(t,r)A_j(r) \right) \circ \left( \Psi_2(s,r) A_j(r) \right) \dif r  \|}_{\mathcal{L}(V^*,V)} \\
&\leq  \sum_{j=1}^n \int^s_0  \sup_{t \in [s,T]} {\| \left( \Psi_1(t,r)A_j(r) \right) \circ \left( \Psi_1(s,r) A_j(r) \right)  - \left( \Psi_2(t,r)A_j(r) \right) \circ \left( \Psi_2(s,r) A_j(r) \right)  \|}_{\mathcal{L}(V^*,V)} \dif r \\
&= \sum_{j=1}^n \int^s_0  \sup_{t \in [s,T]} {\| \left( \Psi_1(t,r)A_j(r) \right) \circ \left( \Psi_1(s,r) A_j(r) \right)  - \left( \Psi_2(t,r)A_j(r) \right) \circ \left( \Psi_2(s,r) A_j(r) \right)} \\
&\phantom{=} \ \ \ \ \ \ \ \ \ \ \ \ \ \ \ \ \ \ \ \ \ \pm \left( \Psi_1(t,r)A_j(r) \right) \circ \left( \Psi_2(s,r) A_j(r) \right)\|_{\mathcal{L}(V^*,V)} \dif r \\
&= \sum_{j=1}^n \int^s_0  \sup_{t \in [s,T]} {\| \left( \Psi_1(t,r)A_j(r) \right) \circ \left[ \left(  \Psi_1(s,r) - \Psi_2(s,r) \right) A_j(r) \right] \|_{\mathcal{L}(V^*,V)} } \\
&\ \ \ \ \ \ \ \ \ \ + \sup_{t \in [s,T]} \| \left[ \left(  \Psi_1(t,r) - \Psi_2(t,r) \right) A_j(r) \right]  \circ \left( \Psi_2(s,r)A_j(r) \right) \|_{\mathcal{L}(V^*,V)} \dif r \\
&\leq \sum_{j=1}^n \int^s_0  \sup_{t \in [s,T]} {\| \Psi_1(t,r) \|_{\mathcal{L}(V^*,V)} \| A_j(r) \|_{V^*} \| \Psi_1(s,r) - \Psi_2(s,r) \|_{\mathcal{L}(V^*,V)} \| A_j(r) \|_{V^*} } \\
&\ \ \ \ \ \ \ \ \ \ + \sup_{t \in [s,T]} \| \Psi_1(t,r) - \Psi_2(t,r)
\|_{\mathcal{L}(V^*,V)} \| A_j(r) \|_{\mathcal{L}(V^*,V)} \| \Psi_2(s,r) \|_{\mathcal{L}(V^*,V)} \| A_j(r) \|_{V^*} \dif r \\
&\leq C(T) \int^s_0 \sup_{t \in [s,T]} \| \left(\Psi_1(s,r) -  \Psi_2(s,r) \right) \|_{\mathcal{L}(V^*,V)} + \sup_{t \in [s,T]} {\| \left( \Psi_1(t,r) - \Psi_2(t,r) \right) \|}_{\mathcal{L}(V^*,V)}\dif r\\
&\leq 2C(T) \int^s_0 \sup_{t \in [s,T]} \| \left(\Psi_1(t,r) -  \Psi_2(t,r) \right) \|_{\mathcal{L}(V^*,V)}\dif r \\
&\leq 2C(T) \int^s_0 \sup_{t \in [r,T]} \| \left(\Psi_1(t,r) -  \Psi_2(t,r) \right) \|_{\mathcal{L}(V^*,V)}\dif r = 2C(T) \int^s_0 \varphi (r) \dif r.
\end{align*}
Now, by the Gronwall lemma we obtain
 $\sup_{t \in [s,T]} \| \Psi_1(t,s) - \Psi_2(t,s) \|_{\mathcal{L}(V^*,V)} = 0$ for all $0 \leq s \leq T$ and the proof is complete.

\end{proof}

\begin{remark}\label{remark}
Note that if the signal $\theta = \lbrace \theta_t, t \in [0,T] \rbrace$ takes its values in the Hilbert space $H$, the family  $\left( A(s) \right)_{s \in [0,T]}$ of observation operators $H \rightarrow \mathbb{R}^{n}$ can be characterised by $n$ functionals from the dual space of $H$ and no embedding is needed. In this case the equation (\ref{main_cov_eq}) can be expressed using adjoint operators $A^*$ and $\Phi^*$ as
\begin{equation} \label{main_cov_eq-hilb}
    \Phi(t,s) = K^{\theta}(t,s) - \int^s_0 \Phi(t,r)A^*(r) A(r)\Phi^*(s,r) \dif r.
\end{equation}
Indeed we have
\begin{align*}
&\langle \left( \sum_{j=1}^n \int^s_0 \left( \Phi(t,r)A_j(r) \right) \circ \left( \Phi(s,r) A_j(r) \right) \dif r \right) h , k \rangle_H\\
&= \sum_{j=1}^n \int^s_0 \langle \Phi(t,r)A_j(r) \langle \Phi(s,r) A_j(r) , h \rangle_H , k \rangle_H \dif r
\\
&= \sum_{j=1}^n \int^s_0 \langle \Phi(t,r)A_j(r) \langle \Phi^*(s,r) h , A_j(r) \rangle_H , k \rangle_H \dif r\\
&= \langle \left( \int^s_0 \Phi(t,r)A^*(r) A(r)\Phi^*(s,r) \dif r \right) h , k \rangle_H
\end{align*}
for all $h,k \in H$ and all $(t,s) \in \Delta$.
\end{remark}

\section{Examples}

\begin{Priklad}
\textit{Linear Stochastic Evolution Equation driven by Wiener Process}\\
Let the signal $\theta = \lbrace \theta_t, t \in [0,T] \rbrace$ be an $H$ - valued random process defined by stochastic evolution equation
\begin{equation}\label{strong_solution_wiener_example}
\dif \theta_t = \mathcal{A} \theta_s \dif s + G \dif \mathcal{W}_t, \ \ \ \theta_0 = 0,
\end{equation}
where $\mathcal{A}$ is the infinitesimal generator of a strongly continuous semigroup $(S(t)) _{t \in \R_+}$ in $H$, $G \in \mathcal{L}(H)$ and $\lbrace \mathcal{W}_t, t \in [0,T] \rbrace$ is an $H$ - valued standard cylindrical Wiener process defined on a stochastic basis $( \Omega , {F}, {P}, ({F}_t) )$.

Assume $\forall t > 0$, $S(t)G \in \mathcal{L}_2 (H)$ and $\int_0^{T_0} \mid S(r)G \mid^2_{\mathcal{L}_2 (H)} \dif r < \infty$ for some $T_0 > 0$. Then the equation (\ref{strong_solution_wiener_example}) has a unique mean - square continuous solution $\lbrace \theta_t, t \in [0,T] \rbrace$ (cf. \cite{ZabPrat}).
The observation $\xi = \lbrace \xi_t, t \in [0,T] \rbrace$ is given by the equation (\ref{observation_process}) where $V=H$, $\lbrace W_t, t \in [0,T] \rbrace$ is a $\R^n$-valued Wiener process on  $( \Omega , {F}, {P}, ({F}_t) )$ independent of $\lbrace \mathcal{W}_t, t \in [0,T] \rbrace$ and $A: [0,T] \rightarrow \mathcal{L}(H,\R^n)$ is strongly measurable and bounded.

For all $0 \leq r \leq t \leq T$ the process $\theta_t$ satisfies
\begin{equation}\label{theta_solution}
\theta_t = S(t-r) \theta_r + \gamma_{t,r},
\end{equation}
where $\gamma_{t,r} = \int_r^t S(t-s) G \dif\mathcal{W}_s$ is a well defined centered Gaussian variable in $H$ and is independent of $\theta_r$ (see \cite{ZabPrat}).

In this case the equations (\ref{main_integral_eq}) and (\ref{main_cov_eq}) simplify to the infinte - dimensional analogue of standard Kalman - Bucy filter as shown in the following theorem.

\begin{Veta}\label{theorem_wiener-KB}
The filter $\widehat{\theta}$ satisfies the stochastic differential equation
  \begin{equation}\label{main_eq_wiener}
  \dif \widehat{\theta}_t = \mathcal{A} \widehat{\theta}_t \dif t + \Phi(t) A^{*}(t) \left( \dif \xi_t - A(t) \widehat{\theta}_t \dif t \right), \ \ t \in [0,T],
  \end{equation}
where the solution to (\ref{main_eq_wiener}) is understood in the mild sense, i.e. $ \widehat{\theta}_t$ solves the equation
\begin{equation}\label{main_eq_wiener_mild}
 \widehat{\theta}_t = \int^t_0 S(t-s) \Phi(s) A^{*}(s) \dif \xi_s -   \int^t_0 S(t-s) \Phi(s) A^{*}(s)A(s) \widehat{\theta}_s \dif s. \\
\end{equation}
The family of operators $\left( \Phi(t) \right)_{t \in [0,T]}$: $H \rightarrow H$ defined as $\Phi(t) =\mathbb{E} \left[(\theta_t -  \widehat{\theta}_t) \circ (\theta_t -  \widehat{\theta}_t) \right]$ for all $t \in [0,T]$ satisfies the differential equation
    \begin{equation*}
    \dot{\Phi}(t) = \mathcal{A}  \Phi(t) +  \Phi(t) \mathcal{A}^{*} -  \Phi(t)A^*(t) A(t)\Phi(t) + Q, \quad Q = G G^*
    \end{equation*}
in the weak sense, that is
\begin{equation}\label{main_cov_eq_wiener_weak}
\dfrac{\dif}{\dif t} \langle \Phi(t) x , y \rangle
 = \langle \Phi(t) x , \mathcal{A}^{*} y \rangle +   \langle \mathcal{A}^{*} x , \Phi(t) y \rangle -  \langle A(t)\Phi(t) x , A(t)\Phi(t) y \rangle + \langle Qx , y \rangle
\end{equation}
for all $x,y \in Dom(\mathcal{A}^*)$.
\end{Veta}
\begin{proof}
Note that the operator $\Phi(t)$ is in fact the operator $\Phi(t,s)$ defined in theorem (\ref{main_theorem}) when $s = t$ and is selfadjoint. Set $\Phi(t) = \Phi(t,t)$ and $K(t) = K^{\theta}(t,t)$ on $t \in [0,T]$.

Using (\ref{main_integral_eq}), (\ref{theta_solution}), the independence of $\gamma_{t,s}$ and $\theta_s$ for all $0 \leq s \leq t$ and the definition of operator $\Phi(t,s)$ in theorem (\ref{main_theorem}) we obtain 
\begin{align*}
&\widehat{\theta}_t 
= \int^t_0 \mathbb{E} \left[ \left( S(t-s) \theta_s + \gamma_{t,s} \right) \circ (\theta_s - \widehat{\theta}_s) \right] A^{*}(s) \dif \xi_s \\
 & \ \ \ \ \ \ \ \ \ \ \ \ \ 
 -  \int^t_0 \mathbb{E} \left[ \left( S(t-s) \theta_s + \gamma_{t,s} \right) \circ (\theta_s - \widehat{\theta}_s) \right] A^{*}(s)A(s) \widehat{\theta}_s \dif s \\
&= \int^t_0 S(t-s) \mathbb{E} \left[ \theta_s \circ (\theta_s - \widehat{\theta}_s) \right] A^{*}(s) \dif \xi_s -  \int^t_0 S(t-s) \mathbb{E} \left[ \theta_s \circ (\theta_s - \widehat{\theta}_s) \right] A^{*}(s)A(s) \widehat{\theta}_s \dif s \\
& =  \int^t_0 S(t-s) \Phi(s) A^{*}(s) \dif \xi_s -   \int^t_0 S(t-s) \Phi(s) A^{*}(s)A(s) \widehat{\theta}_s \dif s \\
\end{align*}
for all $t \in [0,T]$ which is exactly (\ref{main_eq_wiener_mild}).

Using the same arguments as above, equation (\ref{main_cov_eq-hilb}) and self - adjointness of the operator $\Phi(t)$, $t \in [0,T]$ we have
\begin{align*}
&\Phi(t) = \\
& K(t) - \int^t_0 \mathbb{E} \left[ \left( S(t-r) \theta_r + \gamma_{t,r} \right) \circ (\theta_r - \widehat{\theta}_r) \right]  A^*(r) A(r) \left( \mathbb{E} \left[ \left( S(t-r) \theta_r + \gamma_{t,r} \right) \circ (\theta_r - \widehat{\theta}_r) \right] \right)^* \dif r \\
&= K(t) - \int^t_0  S(t-r) \mathbb{E} \left[ \theta_r \circ (\theta_r - \widehat{\theta}_r) \right]  A^*(r) A(r) \left( S(t-r) \mathbb{E} \left[  \theta_r \circ (\theta_r - \widehat{\theta}_r) \right] \right)^* \dif r\\
&= K(t) - \int^t_0 S(t-r) \Phi(r) A^*(r) A(r)\Phi(r) {S^*(t-r)} \dif r .
\end{align*}
Since the covariance operator $K(t), \ t \in [0,T]$  is a mild solution to the equation
\begin{align*}
\dot{K}(t) = \mathcal{A} K(t) + K(t) \mathcal{A}^{*} + Q, \ \ \ \ K(0) = 0,
\end{align*}
which takes the form
\begin{align*}
K(t) = \int_0^t S(t-r) Q S^{*}(t-r) \dif r
\end{align*}
we obtain
\begin{align*}
&\Phi(t) =\int^t_0 S(t-r) \left( Q - \Phi(r) A^*(r) A(r)\Phi(r) \right) {S^*(t-r)} \dif r 
\end{align*}
which is known to be equivalent to the weak form of the equation (\ref{main_cov_eq_wiener_weak}) (cf. \cite{CP}).
\end{proof}

\end{Priklad}

Now, let us consider a signal defined by linear stochastic evolution equation driven by a regular Gauss - Volterra noise. First, we summarise a few basic facts from the infinte - dimensional theory of such processes (see \cite{CoMa}, \citep{CMO}, \cite{Coup} for more details, see also \cite{CMS} for a brief survey).

By scalar Gauss - Volterra process (called also $\alpha$ - regular Volterra process) we understand a centered Gaussian process $\lbrace b_t , t \in [0,T] \rbrace$, $b_0 = 0$, the covariance of which takes the form
\begin{equation*}
\E [b_t b_s] = R(s,t) = \int_0^{s \wedge t} K(s,r)K(t,r) \dif r, \ \ \ s,t \geq 0,
\end{equation*}
where the kernel  $\left( K(t,r) , (t,r) \in \Delta \right)$ satisfies
\begin{align*}
&(i)\ \ \ K(0,0) = 0,\\
&(ii) \ \ \ K( \cdot , r) \in C^1 ([r,T]), \ \ \  0 \leq r \leq T,\\
&(iii) \ \ \ \exists \ \alpha \in \left( 0,\dfrac{1}{2} \right); \ \ \bigg\vert \dfrac{\dif K}{\dif u}(u,r) \bigg\vert < C(u-r)^{\alpha-1}\left( \dfrac{u}{r} \right)^{\alpha}. \ \ \ \ \ \ \ \ \ \ \ \ \ \ \ \ \ \ \ \ \ \ \ \ \ \ \ 
\end{align*}
Note that in virtue of the Kolmogorov continuity criterion the process $\left( b_t , t \in [0,T] \right)$ has an\\
$\epsilon$ - H\"{o}lder version for $\epsilon < \alpha +\dfrac{1}{2}$ \  (cf. \cite{CMO}, Remark 2.1).

An important example of $\alpha$ - regular Gauss - Volterra process is the fractional Brownian motion with the Hurst parameter $h > \dfrac{1}{2}$, in which case $\alpha = h - \dfrac{1}{2}$ and
\begin{equation*}
K(t,s) = K^h(t,s) = C_h \int_s^t \left(\dfrac{u}{s} \right)^{h - \frac{1}{2}} (u-s)^{h-\frac{3}{2}} \dif u,
\end{equation*}
where $c_h$ is a suitable constant (see \cite{alNualart}).

Another example is the multifractional Brownian motion where the Hurst parameter $h = h(t)$ is a function of time. For the definition and conditions on $h$ which ensure that the process satisfies the above assumptions see \cite{CoMa}, Example 2.14.

Now, given a separable Hilbert space $H$ and a stochastic basis $(\Omega , F,P,(F_t) )$, the cylindrical Gauss - Volterra process  $\lbrace B_t , t \in [0,T] \rbrace$ on $H$ is defined by the formal series
\begin{equation}\label{cylindrical_volterra}
B_t = \sum_{n = 1}^{\infty} \beta_n(t) e_n
\end{equation}
where $\lbrace e_n, n \in \N \rbrace$ is an orthonormal basis in $H$ and $\lbrace \beta_n (t) , t \in [0,T] \rbrace_{n \in \N}$ is a sequence of independent scalar $\alpha$ - regular Gauss - Volterra processes with the same kernel $K$. The series (\ref{cylindrical_volterra}) does not converge in the space $H$ but defines the system of scalar processes $\lbrace B(h) , h \in H \rbrace$, 
$$B_t (h) = \sum_{n=1}^{\infty} \langle e_n, h \rangle \beta_n (t)$$ (for definitions and basic properties of cylindrical Volterra processes and stochastic integrals driven by them see \cite{CoMa}, \cite{CMO}).

Suppose that the signal satisfies the equation
\begin{equation}\label{Volterra_process-signal}
\dif \theta_t = \mathcal{A} \theta_s \dif s + G \dif B_t, \ \ \ t \in [0,T], \ \ \ \theta_0 = 0,
\end{equation}
where $\mathcal{A}$ is infinitesimal generator of a strongly continuous semigroup $(S(t)) _{t \in \R+}$ in $H$ and $G \in \mathcal{L}(H)$.

By the analyticity of semigroup $S$ there exists $\lambda \in \R$ such that the operator $(\lambda I -  \mathcal{A})$ is strictly positive. Therefore, for  $\delta \geq 0$, we can define the Hilbert space
\begin{equation*}
V_{\delta} = Dom((\lambda I -  \mathcal{A})^{\delta})
\end{equation*}
equipped with the graph norm topology.

The solution $\lbrace \theta_t, t \in [0,T] \rbrace$ to the equation (\ref{Volterra_process-signal}) is understood in the mild sense, that is,
\begin{equation*}
\theta_t = \int_0^t S(t-r) G \dif B_r, \quad t \in [0,T].
\end{equation*}

As a particular case of Corollary 4.1. in \cite{CMO} we obtain the following statement.
\begin{Veta}
Assume $S(u)G \in \mathcal{L}_2(H)$, $u \in [0,T]$, and let there exist a $\gamma \in [0, \alpha + 1/2)$ such that 
\begin{equation}\label{condition_op_G}
\mid S(u)G \mid_{\mathcal{L}_2(H)} \leq cu^{-\gamma}, \quad u\in [0,T]
\end{equation}
for a constant $c>0$. Then $\lbrace \theta_t, t \in [0,T] \rbrace$ has a continuou version in the space $V_{\delta}$ for
\begin{equation}\label{condition_alpha,gamma,delta}
0 \leq \delta < \alpha + \frac{1}{2} - \gamma.
\end{equation}
If $G \in \mathcal{L}_2(H)$ (which corresponds to the case when the driving noise in (\ref{Volterra_process-signal}) may be represented by a genuine $H$-valued Gauss-Volterra process), the condition (\ref{condition_op_G}) is satisfied with $\gamma = 0$ and (\ref{condition_alpha,gamma,delta}) reads $0 \leq \delta < \alpha + 1/2$ (for the fractional Brownian motion $0 \leq \delta < h$).
\end{Veta}

Note that for fractional Brownian motion the proposition holds true even if $h < \frac{1}{2}$ (cf. \cite{DMP}).

\begin{Priklad}\label{final_example}
Consider the signal given by the following parabolic equation 
\begin{equation}\label{volterra_parabolic_eq}
\partial_t u = L_{2m} u + \eta
\end{equation}
on $[0,T] \ \times \ \mathcal{D}$ with initial condition $u(0, \cdot) = 0$ and with the Dirichlet boundary condition
\begin{equation*}
\dfrac{\partial^k u}{\partial \mathbf{x}^k}\bigg\vert_{[0,T] \ \times \ \partial \mathcal{D}} = 0,
\quad k \in \lbrace0, \cdots , m-1 \rbrace,
\end{equation*}
where $\dfrac{\partial}{\partial \mathbf{x}^k}$ denotes the conormal derivative. The domain $\mathcal{D} \subset \R^d$ is open and bounded with smooth boundary and $L_{2m}$ is a differential operator uniformly elliptic of order $2m$,
\begin{equation}\label{Laplace_operator}
L_{2m} = \sum_{|k| \leq 2m} a_k(\cdot) \partial^k
\end{equation}
with $a_k \in C^{\infty}_b(\mathcal{D})$. The noise $\eta$ is Gauss - Volterra in time and may be white (if $G$ is identity operator) or correlated in space. This system can be reformulated as the stochastic evolution equation (\ref{Volterra_process-signal}) in $H = L^2(\mathcal{D})$. Indeed, the noise $\eta$ is formally given as
\begin{equation*}
\eta(t,\cdot) = G \dfrac{\dif}{\dif t} B_t,
\end{equation*}
where $\lbrace B_t , t \in [0,T] \rbrace$ is a Gauss - Volterra process on $H$ and $\mathcal{A} = L_{2m}\big\vert_{Dom(\mathcal{A})}$ where
\begin{equation}\label{Laplace_operator_domain}
Dom(\mathcal{A}) = \bigg\lbrace f \in W^{2m,2}(\mathcal{D}) : \dfrac{\partial^k f}{\partial \mathbf{x}^k} = 0\ \text{on} \ \partial \mathcal{D}\ \text{for} \ k \in  \lbrace0, \cdots , m-1 \rbrace \bigg\rbrace.
\end{equation}
The operator $\mathcal{A}$ generates an analytic semigroup $(S(t), t\in [0,T])$ on $H$.

If the observation is given by the equation (\ref{observation_process}) where $V=H=L^2(\mathcal{D})$, $\lbrace W_t, t \in [0,T] \rbrace$ is independent of $\lbrace B_t, t \in [0,T] \rbrace$ and $A: [0,T] \rightarrow \mathcal{L}(L^2(\mathcal{D}),\R^n)$ is strongly measurable and bounded, Theorems \ref{main_theorem}, \ref{uniqueness_theorem} and Remark \ref{remark} may be applied.

However, it may be interesting to consider the case when the only accessible information comes from observation of the signal at given points $z_1, \ldots, z_n \in \mathcal{D}$. Then the signal process has to be more regular.

For simplicity, assume that $m = 1$, i.e. (\ref{volterra_parabolic_eq}) is a stochastic heat equation. If the condition (\ref{condition_alpha,gamma,delta}) is satisfied and, moreover 
\begin{equation*}
\delta > \dfrac{d}{4}
\end{equation*}
then the signal $\theta$ has a continuous version in $V_\delta$ (cf. \cite{CoMa}, \cite{CMO}) that is continuously embedded into $C(\mathcal{D})$ by the Sobolev theorem. This follows from the fact that $V_\delta \subset W^{2\delta,2}(\mathcal{D})$ (cf.  \cite{VdeltaEmbedd}). Such a choice of $\delta$ is possible if
\begin{equation}\label{C_embedding_crit}
\alpha + \frac{1}{2} - \gamma > \frac{d}{4}
\end{equation}
(note that if $G \in \mathcal{L}_2(H)$ we may put $\gamma = 0$ and, on the other hand, for $\gamma = d/4$ we may consider arbitrary $G \in \mathcal{L}(H)$).

Then in Theorems \ref{main_theorem} and \ref{uniqueness_theorem} we may put $V = V_\delta \hookrightarrow C(\mathcal{D})$ and as an example of observation operator we may take $A(t) = A$ defined as
\begin{equation}\label{pointwise_operator}
A \theta_t = (\theta_t(z_1), \dots , \theta_t(z_n)),
\end{equation}
where $z_1, \dots ,z_n \in \mathcal{D}$, which corresponds to pointwise observation of the signal process $\theta$.

In this case the equations (\ref{main_integral_eq}) and (\ref{main_cov_eq}) can be rewritten according to the following theorem.

\begin{Dusledek}\label{final_corollary}
Let the signal $\theta = \lbrace \theta_t, t \in [0,T] \rbrace$ satisfy the stochastic evolution equation
\begin{equation*}
\dif \theta_t = \mathcal{A} \theta_s \dif s + G \dif B_t, \ \ \ \theta_0 = 0,
\end{equation*}
where $\lbrace B_t , t \in [0,T] \rbrace$ is a Gauss - Volterra process on $L^2(\mathcal{D})$, $G \in \mathcal{L}(L^2(\mathcal{D}))$ and $\mathcal{A} = L_{2}\big\vert_{Dom(\mathcal{A})}$ is given by (\ref{Laplace_operator}) and (\ref{Laplace_operator_domain}). Further assume that condition (\ref{C_embedding_crit}) holds. Consider the observation process $\xi = \lbrace \xi_t, t \in [0,T] \rbrace$ given by (\ref{observation_process}) with operator $A(t) = A$ defined by (\ref{pointwise_operator}).
Then the filter $\widehat{\theta}$ satisfies stochastic integral equation
\begin{equation}\label{main_integral_eq_pointwise}
  \widehat{\theta}_t = \sum_{j=1}^n \int^t_0 \Phi_{z_j}(t,s) \dif \xi^j_s -  \sum_{j=1}^n \int^t_0 \Phi_{z_j}(t,s) \widehat{\theta}_s (z_j) \dif s, \ \ t \in [0,T],
\end{equation}
where $\Phi_{z_i}$: $\Delta \rightarrow C(\mathcal{D})$ is defined as $\Phi_{z_i}(t,s) = \mathbb{E}[(\theta_s -  \widehat{\theta}_s)(z_i) \theta_t]$ for all $(t,s) \in \Delta$, $i = 1, \dots , n$ and integral equation
\begin{equation} \label{main_cov__pointwise}
    \Phi_{z_i}(t,s) = \E [\theta_s(z_i) \theta_t] - \sum_{j=1}^n \int^s_0 \Phi_{z_j}(s,r)(z_i) \Phi_{z_j}(t,r) \dif r, \quad i = 1, \dots , n
\end{equation}
is satisfied.
\end{Dusledek}

\begin{proof}
From (\ref{main_integral_eq}), the definition of operator $\Phi$, the continuous embedding $V_{\delta} \hookrightarrow C(\mathcal{D})$ and (\ref{pointwise_operator}) we have
\begin{align*}
&\widehat{\theta}_t = \int^t_0 \Phi(t,s)A^{*} \dif \xi_s -  \int^t_0 \Phi(t,s)A^{*}A \widehat{\theta}_s \dif s\\
&= \sum_{j=1}^n  \int^t_0 \E \left[\theta_t \langle \theta_s -  \widehat{\theta}_s , A_j \rangle_\bb \right] \dif \xi_s - \sum_{j=1}^n \int^t_0 \E \left[\theta_t \langle \theta_s -  \widehat{\theta}_s, A_j \rangle_\bb \right] \langle \widehat{\theta}_s , A_j \rangle_\bb \dif s\\
&=  \sum_{j=1}^n  \int^t_0 \E \left[(\theta_s -  \widehat{\theta}_s) (z_j) \theta_t \right] \dif \xi_s - \sum_{j=1}^n \int^t_0 \E \left[(\theta_s -  \widehat{\theta}_s) (z_j) \theta_t \right] \widehat{\theta}_s (z_j) \dif s
\end{align*}
which concludes the proof of (\ref{main_integral_eq_pointwise}).

Analogously, using (\ref{main_cov_eq}) we obtain
\begin{align*}
&\Phi_{z_i}(t,s) = \Phi(t,s) A_i\\
&=  \E \left[ \theta_t \langle \theta_s , A_i \rangle_\bb \right] - \sum_{j=1}^n \int^s_0 \E \left[\theta_t \langle \theta_r -  \widehat{\theta}_r , A_j \rangle_\bb \right] \langle \E \left[\theta_s \langle \theta_r -  \widehat{\theta}_r , A_j \rangle_\bb \right] , A_i \rangle_\bb \dif r\\
&= \E [\theta_s(z_i) \theta_t] - \sum_{j=1}^n \int^s_0 \E \left[ (\theta_r -  \widehat{\theta}_r ) (z_j) \theta_t \right] \left[ \left( \E \left[(\theta_r -  \widehat{\theta}_r ) (z_j) \theta_s \right] \right)(z_i) \right] \dif r\\
&=  \E [\theta_s(z_i) \theta_t]- \sum_{j=1}^n \int^s_0 \Phi_{z_j}(t,r) \left[ \Phi_{z_j}(s,r) (z_i ) \right] \dif r
\end{align*}
which concludes the proof of (\ref{main_cov__pointwise}).
\end{proof}

Note that if the driving process $B=B^h$ is a fractional Brownian motion the condition (\ref{C_embedding_crit}) reads 
\begin{equation*}
 h > \frac{d}{4} + \gamma, \quad h \in (0,1),
\end{equation*}
(cf. \citep{CMO}).

It may be interesting to specify the covariances $\E [\theta_s(z_i) \theta_t]$ that appear in the equation (\ref{main_cov__pointwise}). Suppose, for simplicity, that the driving process $B_t = B^h_t$ is a fractional Brownian motion with the Hurst parameter $h > 1/2$, $n=1$ (i.e. the process $\theta = \lbrace \theta_t, t \in [0,T] \rbrace$ is observed at a single point $z_1 \in \mathcal{D}$) and the noise term $G$ is Hilbert-Schmidt, i.e. it may be expressed as
\begin{equation*}
[G(f)](\xi) = \int_{\mathcal{D}} k(\xi,\eta)f(\eta) \dif \eta, \quad f \in H = L^2(\mathcal{D}),
\end{equation*}
where $k \in L^2 (\mathcal{D} \times \mathcal{D})$. It is also well known that the semigroup $(S(t), t\in \R)$ may be represented by a Green function $g: [0,T] \times \mathcal{D} \times \mathcal{D} \rightarrow \R$, that is,
\begin{equation*}
[S(t)(f)](\xi) = \int_{\mathcal{D}} g(t,\xi,\eta)f(\eta) \dif \eta, \quad f \in H, \ t>0.
\end{equation*}
Therefore, the composition $S(t)G$ may be written as
\begin{equation*}
[S(t)G(f)](\xi) = \int_{\mathcal{D}} \widetilde{g}(t,\xi,\eta)f(\eta) \dif \eta, \quad f \in H, \ t>0,
\end{equation*}
where the composition kernel $\widetilde{g}$ is given by
\begin{equation*}
\widetilde{g}(t,\xi,\eta) = \int_{\mathcal{D}} g(t,\xi,\lambda)k(\lambda,\eta) \dif \lambda.
\end{equation*}
Now it is standard to compute the covariance 
\begin{equation*}
 \E \left[ \theta_s(z_1) \theta_t \right] (\eta) = \int_0^s \int_0^t \phi_h(\lambda,r) \int_{\mathcal{D}} \widetilde{g}(s-r,z_1,\xi) \widetilde{g}(t-\lambda,\eta,\xi) \dif \xi \dif \lambda \dif r,
\end{equation*}
for $(s,t) \in \Delta$, $\eta \in \mathcal{D}$, where $\phi_h(\lambda,r) = h(2h-1) \mid \lambda - r \mid^{2h-2}$ and $(\eta,\lambda) \in \Delta$.

\end{Priklad}

\vspace{1cm}
\par
\noindent \textsc{Acknowledgements:} \it{This research was partially suported by GAUK Grant no. 980218, Czech Science Foundation (GA\v CR) Grant no. 19-07140S  and by the SVV Grant No. 260454.}

\bibliographystyle{mybst}
\bibliography{references}

\end{document}